\documentclass[10pt,onesite,draft]{article}
\newfam\cyrfam

\usepackage{amssymb,latexsym,amscd,amsmath,amsfonts, amsthm,enumerate}

\newtheorem{theorem}{Theorem}[section]

\newtheorem{corollary}[theorem]{Corollary}
\newtheoremstyle{definition}% name
  {6pt}%      Space above
  {6pt}%      Space below
  {}%         Body font
  {}%         Indent amount (empty = no indent, \parindent = para indent)
  {\bfseries}% Thm head font
  {.}%        Punctuation after thm head
  {.5em}%     Space after thm head: " " = normal interword space;
  {}%

\theoremstyle{definition}

\newtheoremstyle{remark}% name
  {6pt}%      Space above
  {6pt}%      Space below
  {}%         Body font
  {}%         Indent amount (empty = no indent, \parindent = para indent)
  {\bfseries}% Thm head font
  {.}%        Punctuation after thm head
  {.5em}%     Space after thm head: " " = normal interword space;
  {}%

\theoremstyle{remark}
\newtheorem{remark}[theorem]{Remark}
\newtheoremstyle{example}% name
  {6pt}%      Space above
  {6pt}%      Space below
  {}%         Body font
  {}%         Indent amount (empty = no indent, \parindent = para indent)
  {\bfseries}% Thm head font
  {.}%        Punctuation after thm head
  {.5em}%     Space after thm head: " " = normal interword space;
  {}%

\theoremstyle{example}
\newtheorem{example}[theorem]{Example}

\makeatletter
\renewcommand\@makefntext[1]{%
\setlength\parindent{1em}%
\noindent \makebox[1.8em][r]{}{#1}} \makeatother

\begin{document}
\parskip 4pt
\large \setlength{\baselineskip}{15 truept}
\setlength{\oddsidemargin} {0.5in} \overfullrule=0mm
\def\bfh{\vhtimeb}
\date{}

\title{\bf \large  ON SOME PROPERTIES OF PARTIAL SUMS  
 }
\def\b{\vntime}

\author{
 Duong Quoc Viet and Truong Thi Hong Thanh \\
\small Department of Mathematics, Hanoi National University of Education\\
\small 136 Xuan Thuy Street, Hanoi, Vietnam\\
\small Emails: duongquocviet@fmail.vnn.vn \;and\; thanhtth@hnue.edu.vn \\}
 \date{}

\maketitle \centerline{
\parbox[c]{9. cm}{
\small {\bf  ABSTRACT:} The paper gives some criteria for  partial
sums of rational number sequences  to be not 
rational functions and  to be not algebraic functions.
 As an application, we study  partial sums of some famous rational number sequences in
mathematical analysis with relevance to number theory. 
  }}
\section{Introduction}
\noindent To determine properties of  partial sums $S(n) = \sum_{i=1}^nu_i$  of
number sequences $(u_i)_{i\geqslant 1}$ is usually an interesting
problem. Although one has given different tools for this problem, but it is hard to know that  
functions for partial sums  are algebraic
functions or transcendental functions.

In this paper, we give some criteria for functions in one positive
integer variable to be not 
rational functions (see Theorem \ref{th2.1}) and to be not  algebraic
functions  (see Theorem \ref{th2.0}). In particular, we get the
following corollaries for partial sums of rational number
sequences (see Section 2).
\footnotetext{\begin{itemize}
 \item[ ]{\it Mathematics Subject
Classification}(2010): Primary  11B83.    Secondary 12B31; 11B34.
\item[ ]{\it Key words and phrases}: Partial sums; number
sequences; transcendental function.
\end{itemize}}
Let $(u_i)_{i\geqslant 1}$ be a rational number sequence. Then
$S(n) = \sum_{i=1}^nu_i$ is not a rational function in $n$ if any
one of the following conditions holds (see Corollary \ref{co2.3}):

\begin{itemize}
\item[$\mathrm{(a)}$] $\lim_{n \rightarrow \infty}S(n) = \infty$
and $\lim_{n \rightarrow \infty}\frac{S(n)}{n} = 0.$

\item[$\mathrm{(b)}$] $\lim_{n \rightarrow \infty}S(n) = 0$ and
$\lim_{n \rightarrow \infty}nS(n) = \infty.$

\item[$\mathrm{(c)}$] $\lim_{n \rightarrow \infty}S(n)$ is an
irrational number.
\end{itemize}
And if $\lim_{n \rightarrow \infty}S(n)$ is a transcendental
number then $S(n) = \sum_{i=1}^nu_i$ is a transcendental function
in $n$ (see Corollary \ref{co2.4}).

Using these criteria, we study  
 partial sums of some famous rational number sequences in
mathematical analysis with relevance to number theory
(see Section 3).

\section{Some Criteria for Partial Sums}
\noindent This section gives some criteria for functions in one
positive integer variable to be not rational functions and algebraic functions.

 Denote by $\mathbb{Q};$
 $\mathbb{C}$ the field of the rational numbers;  the field of the complex numbers,
respectively. Let $K$ be a subfield of $\mathbb{C}$; $\alpha \in
\mathbb{C}.$ Let $f(n)$ be a complex-valued function in one
positive integer variable $n.$ Then one can make the following
definitions:

\begin{itemize}

\item[$\mathrm{(i)}$] $\alpha$ is a {\it rational number} if
$\alpha \in \mathbb{Q}[i].$ And $\alpha$ is an {\it irrational
number} if $ \alpha \notin \mathbb{Q}[i].$

\item[$\mathrm{(ii)}$] $\alpha$ is  {\it algebraic} over $K$  if
there exist $c_0,\ldots, c_h \in K$ with $c_h \ne 0$ such that
$$c_h \alpha^h + c_{h-1} \alpha^{h-1} +\cdots+c_1\alpha + c_0 = 0.$$ And $\alpha$ is a {\it
transcendental number over $K$} if it is not an algebraic number over $K$.

\item[$\mathrm{(iii)}$] $\alpha$ is an {\it algebraic number } if
$\alpha$ is algebraic over $\mathbb{Q}.$ And $\alpha$ is a {\it
transcendental number } if it is not an algebraic number.
Denote by $\overline{\mathbb{Q}}$
the field of the algebraic numbers.
 \item[$\mathrm{(iv)}$] $f(n)$ is
a {\it rational function} over $K$ if $$f(n) = \frac{a_m n^m +
a_{m-1} n^{m-1}+\cdots + a_1n + a_0 }{b_d n^d + b_{d-1}
n^{d-1}+\cdots + b_1n + b_0}$$ where $a_0.\ldots, a_m; b_0,\ldots,
b_d \in K$ with $b_d \ne 0.$

\item[$\mathrm{(v)}$] $f(n)$ is an {\it algebraic function} over
$K$ if there exist polynomials in $n$ with coefficients in $K$:
$q_0(n),\ldots, q_k(n) \in K[n]$ with $q_k(n) \ne 0$ such that
$$ q_k(n)[f(n)]^k + q_{k-1}(n)[f(n)]^{k-1}+\cdots+q_1(n)[f(n)] +
q_0(n) = 0.$$
And
$f(n)$ is a {\it transcendental function}
over $K$ if it is not an algebraic function $\mathrm{over}\; K.$
\end{itemize}

Then we obtain the following results.

\begin{theorem}\label{th2.1} Let $f(n)$ be a function in one
positive integer variable $n.$ Then the following statements hold.

\begin{itemize}
\item[$\mathrm{(i)}$] If $\lim_{n \rightarrow \infty}f(n) =
\infty$ and $\lim_{n \rightarrow \infty}\frac{f(n)}{n} = 0$ then
$f(n)$ is not a rational function in $n$ over $\mathbb{C}.$

\item[$\mathrm{(ii)}$] If $\lim_{n \rightarrow \infty}f(n) = 0$
and $\lim_{n \rightarrow \infty}nf(n) = \infty$ then $f(n)$ is not
a rational function in $n$ over $\mathbb{C}.$
\item[$\mathrm{(iii)}$] If $\lim_{n \rightarrow \infty}f(n)$ is an
irrational number then $f(n)$ is not a rational function in $n$
over $\mathbb{Q}[i].$
\end{itemize}
\end{theorem}
\begin{proof} The proof of (i): If $f(n)$ is
a rational function over $\mathbb{C}$ then $f(n) =
\frac{P(n)}{Q(n)}$ here $P(n)$ and $Q(n)$ are two polynomials in
$n$ over $\mathbb{C}.$ Then on the one hand, since $\lim_{n
\rightarrow \infty}f(n) = \infty,$ it follows that $\deg P(n)  >
\deg Q(n).$ On the other hand, since $\lim_{n \rightarrow
\infty}\frac{f(n)}{n} = \lim_{n \rightarrow
\infty}\frac{P(n)}{nQ(n)} = 0,$ $\deg nQ(n)  > \deg P(n).$ From
the above facts we get $$\deg P(n)  \ge \deg Q(n)+1= \deg nQ(n) >
\deg P(n),$$ contradiction.  Arguing similarly with the proof of
(i), we obtain (ii). The proof of (iii): Suppose that $f(n)$ is a
rational function over $\mathbb{Q}[i].$ Since $\lim_{n \rightarrow
\infty}f(n)= c$ is an irrational number, $c \ne 0.$ Consequently,
 we have $$f(n) = \frac{a_m n^m + a_{m-1} n^{m-1}+\cdots + a_1n +
a_0 }{b_m n^m + b_{m-1} n^{m-1}+\cdots + b_1n + b_0}$$ where
$a_0,\ldots, a_m; b_0,\ldots, b_m \in \mathbb{Q}[i]$ with $a_mb_m
\ne 0.$ Note that
$$\lim_{n \rightarrow \infty}\frac{a_m n^m + a_{m-1} n^{m-1}+\cdots + a_1n +
a_0 }{b_m n^m + b_{m-1} n^{m-1}+\cdots + b_1n + b_0}=
\frac{a_m}{b_m} \in \mathbb{Q}[i].$$ So $c \in \mathbb{Q}[i].$
This contradiction shows that $f(n)$ is not a rational function in
$n$ over $\mathbb{Q}[i].$
\end{proof}

\begin{remark} The reverse of Theorem \ref{th2.1} does not hold. For (i):  $f(n) = n ^{\frac{4}{3}}$
is not a rational function and $\lim_{n \rightarrow \infty}f(n) = \infty,$ but $\lim_{n \rightarrow
\infty}\frac{f(n)}{n} = \infty \ne 0.$ For (ii):
 $g(n)= n ^{-\frac{5}{4}}$ is not a rational function for which $\lim_{n \rightarrow
\infty}g(n)= 0,$ but $\lim_{n \rightarrow
\infty}ng(n)= 0.$ For (iii): $q(n) = 2^{-n}$ is not a rational function, but
$\lim_{n \rightarrow \infty}q(n) = 0 \in \mathbb{Q}.$
\end{remark}

Let $(u_i)_{i\geqslant 1}$ be a rational number sequence. Then
$S(n) = \sum_{i=1}^nu_i$ is a function in $n.$  Hence from Theorem
\ref{th2.1} we immediately get the
following  corollary.

\begin{corollary}\label{co2.3} Let $(u_i)_{i\geqslant 1}$ be a rational number sequence. Then
$S(n)$ is not a rational function in $n$ over $\mathbb{Q}[i]$ if
any one of the following conditions holds:

\begin{itemize}
\item[$\mathrm{(i)}$] $\lim_{n \rightarrow \infty}S(n) = \infty$
and $\lim_{n \rightarrow \infty}\frac{S(n)}{n} = 0.$

\item[$\mathrm{(ii)}$] $\lim_{n \rightarrow \infty}S(n) = 0$ and
$\lim_{n \rightarrow \infty}nS(n) = \infty.$

\item[$\mathrm{(iii)}$] $\lim_{n \rightarrow \infty}S(n)$ is an
irrational number.
\end{itemize}

\end{corollary}

\begin{remark} In general,  the reverse of Corollary \ref{co2.3} does not hold. Indeed, for (i): 
Choose $(u_i= i.i!)_{i\geqslant 1},$ then
$$S(n) = \sum_{i=1}^n i.i! = (n +1)! -1$$ 
is not a rational function and $\lim_{n \rightarrow \infty}S(n) = \infty,$ but $\lim_{n \rightarrow
\infty}\frac{S(n)}{n} = \infty \ne 0.$ For (ii): 
Let $(u_i)_{i\geqslant 1}$ be a rational number sequence
defined by  $u_1 = -\frac{1}{2}$
and $u_i = \frac{i}{(i+1)!}$ for all $i \ge 2.$
Then 
 $S(n)=\sum_{i=1}^n\frac{i}{(i+1)!} =  - \frac{1}{(n+1)!} $ is not a rational function for which $\lim_{n \rightarrow
\infty}S(n)= 0,$ but $\lim_{n \rightarrow
\infty}nS(n)= 0 \ne \infty.$ For (iii): $S(n) = \sum_{i=1}^n2^{-i} = 1-2^{-n}$ is not a rational function, but
$\lim_{n \rightarrow \infty}S(n) = 1 \in \mathbb{Q}.$
\end{remark}

\begin{theorem}\label{th2.0} Let $K$ be a subfield of $\mathbb{C},$ and let $f(n)$ be a function in one
positive integer variable $n.$  Assume that $\lim_{n
\rightarrow \infty}f(n)$ is a transcendental number over $K$. Then $f(n)$
is a transcendental function in $n$ over $K$.
\end{theorem}
\begin{proof}
Arguing by contradiction, suppose that $f(n)$ is a algebraic
function over $K.$  Then there exist
polynomials in $n$ with  coefficients in over $K$:
 $q_0(n),\ldots, q_k(n) \in K[n]$ with
$q_k(n) \ne 0$ such that
$$ q_k(n)[f(n)]^k + q_{k-1}(n)[f(n)]^{k-1}+\cdots+q_1(n)[f(n)] +
q_0(n) = 0.$$ Set $d = \max \{\deg q_0(n),\ldots, \deg q_k(n)\}.$
 It is easily seen that $\lim_{n \rightarrow
\infty}\frac{q_i(n)}{n^d} \in  K$ for all $0
\le i \le k.$ Set $\lim_{n \rightarrow \infty}f(n)= \alpha$ and
$\lim_{n \rightarrow \infty}\frac{q_i(n)}{n^d} = c_i $ for all $0
\le i \le k.$ Then $\alpha$ is a transcendental number over $K$ and
$c_0,\ldots, c_k \in K$ are not all zero.
Since
\begin{align*}
 0 &= \lim_{n \rightarrow \infty}\big( [\frac{q_k(n)}{n^d}][f(n)]^k +
[\frac{q_{k-1}(n)}{n^d}][f(n)]^{k-1}+\cdots+[\frac{q_1(n)}{n^d}][f(n)]
+ [\frac{q_0(n)}{n^d}]\big)\\ &= \lim_{n \rightarrow \infty}
[\frac{q_k(n)}{n^d}]\lim_{n \rightarrow \infty}[f(n)]^k + \lim_{n
\rightarrow \infty}[\frac{q_{k-1}(n)}{n^d}]\lim_{n \rightarrow
\infty}[f(n)]^{k-1}\\
&\;\;\;\;\;\;\;\;\;\;\;\;\;+\cdots+\lim_{n \rightarrow
\infty}[\frac{q_1(n)}{n^d}]\lim_{n \rightarrow \infty}[f(n)] +
\lim_{n \rightarrow \infty}[\frac{q_0(n)}{n^d}]\\
&= c_k \alpha^k + c_{k-1} \alpha^{k-1}+\cdots+c_1\alpha + c_0,
\end{align*}
we obtain  $c_k \alpha^k + c_{k-1} \alpha^{k-1}+\cdots+c_1\alpha +
c_0 = 0$ for $c_0,\ldots, c_k \in K$, not all
zero; i.e., $\alpha$ is algebraic over $K$.
  This contradiction concludes the proof.
\end{proof}

Remember that $\overline{\mathbb{Q}}$ is algebraic over
$\mathbb{Q},$ $\alpha \in \mathbb{C}$ is algebraic over $\overline{\mathbb{Q}}$ if and only if  $\alpha$ is an algebraic number  (see e.g. \cite{SL, PM}). Consequently, we immediately have the following consequence by Theorem \ref{th2.0}.    
\begin{corollary}\label{co2.2} Let  $f(n)$ be a function in one
positive integer variable $n.$  Assume that $\lim_{n
\rightarrow \infty}f(n)$ is a transcendental number. Then $f(n)$
is a transcendental function in $n$ over $\overline{\mathbb{Q}}$.
\end{corollary}

\begin{remark}
 $F(n)= 1+3^{-n}$ is
a transcendental function in $n$ over $\overline{\mathbb{Q}}$ 
which tends to $1$  
(an algebraic number). Hence the reverse of Corollary \ref{co2.2} does not hold.
\end{remark}

Now, assume that $(u_i)_{i\geqslant 1}$ is a rational number sequence and  
$S(n) = \sum_{i=1}^nu_i$. Then $S(n)$ is a function in $n.$  Hence as an immediate consequence of Corollary \ref{co2.2} we  obtain the
following  result.

\begin{corollary}\label{co2.4} Let $(u_i)_{i\geqslant 1}$ be a rational number
sequence.  Assume that  $\lim_{n \rightarrow \infty}S(n)$ is a
transcendental number. Then $S(n) = \sum_{i=1}^nu_i$ is  a
transcendental function in $n$ over $\overline{\mathbb{Q}}$.
 \end{corollary}
\begin{remark} The reverse of Corollary \ref{co2.4} does not hold. Indeed, now we 
choose $(u_i= 3^{-i})_{i\geqslant 1}$, then 
$S(n) = \sum_{i=1}^n3^{-i} = \frac{1}{2}(1- 3^{-n})$ is  a
transcendental function in $n$ over $\overline{\mathbb{Q}},$ 
however $\lim_{n \rightarrow \infty}S(n)= \frac{1}{2}$ is an algebraic number.
\end{remark}
\section{ Some Applications}
\noindent In this section, we give some applications of Section 2
for partial sums of famous rational number sequences in
mathematical analysis with relevance to number theory.
\begin{example}
%\noindent {\bf Example 3.1:}
  Let $a$ and $b$ be non-negative
integers with $a \ne 0$. Then  $$S(n) =
\sum_{j=1}^n\frac{1}{aj+b}$$ is not a rational function in $n$
over $\mathbb{Q}.$
\end{example}

\begin{proof} Set $c = \max \{a, b\}.$ Since $$S(n) = \sum_{j=1}^n\frac{1}{aj+b}\ge
\sum_{j=1}^n\frac{1}{cj+c} =
\frac{1}{c}\sum_{j=1}^n\frac{1}{j+1}$$ and $\lim_{n \rightarrow
\infty}\sum_{j=1}^n\frac{1}{j+1} = \infty$ (see e.g. \cite{AK}),
it follows that $\lim_{n \rightarrow \infty}S(n) = \infty.$ Now,
we will prove that
$$\sum_{j=1}^n\frac{1}{j} < \sqrt{2n+1}$$ by induction on $n.$ The
result is true for $n=1$ since $1 < \sqrt{3}.$ Next assume  that
$n>1.$ By the inductive assumption  we have
$$\sum_{j=1}^n\frac{1}{j} = \sum_{j=1}^{n-1}\frac{1}{j}+\frac{1}{n}
< \sqrt{2(n-1)+1}+\frac{1}{n}.$$ Since
 $\sqrt{2n+1}-\sqrt{2(n-1)+1} =
 \frac{2}{\sqrt{2n+1}+\sqrt{2(n-1)+1}},$ it follows that
$$\sqrt{2(n-1)+1}+\frac{1}{n} < \sqrt{2n+1}$$ if and only if $2n >
\sqrt{2n+1}+\sqrt{2(n-1)+1}.$ This is equivalent to $4n^3(n-2)+1
> 0.$ Remember that $n \ge 2,$  $4n^3(n-2)+1
> 0.$
 So $2n >
\sqrt{2n+1}+\sqrt{2(n-1)+1}$ for all $n \ge 2.$    Induction is
complete. We obtain $S(n) \le \sum_{j=1}^n\frac{1}{j} <
\sqrt{2n+1}.$ From this it follows that $\lim_{n \rightarrow
\infty}\frac{S(n)}{n} = 0.$ Thus, $S(n)$ is not a rational
function over $\mathbb{Q}$ by Corollary \ref{co2.3}(i).
\end{proof}

\noindent{\bf Note 1:} In mathematics, the $n$-th harmonic number
is the sum of the reciprocals of the first $n$ natural numbers:
$H_n = \sum_{j=1}^n\frac{1}{j}.$ Example 3.1 showed that $H_n$ is
not a rational function in $n$ over $\mathbb{Q}.$       
Euler found in
1734 that $\lim_{n \rightarrow \infty}[H_n-\ln n]$ is a constant
called the Euler's constant and one denoted by $\gamma$ this
constant (see e.g. \cite {EM,JS1, WA}). $\gamma = 0.577215665...$
has not yet been proven to be transcendental or even irrational.
Sondow  gave  in 2002 \cite{JS2} criteria for irrationality of
$\gamma.$ Now, we  need to talk more about $H_n.$ 
Denote by  $\psi(x)$ the digamma function, that is, the logarithmic derivative of Euler's $\Gamma$-function. 
Then we have $H_n = \psi(n+1)+ \gamma$ (see e.g. \cite{MN}).
By Murty and Saradha in 2007 \cite[Theorem 1] {MN},    $\psi(x) + \gamma$ takes transcendental values at infinitely many   $x$  rational. Consequently, $\psi(x) + \gamma$ is a transcendental function over $\overline{\mathbb{Q}}.$ Hence we would like to     give a conjecture that 
"$H_n = \psi(n+1)+ \gamma$ {\it is a transcendental function over $\overline{\mathbb{Q}}$"}.

\begin{example} 
Let $(u_n)_{n \ge 1}$ be a sequence
  defined by $u_1 = 1$ and $$u_{i+1} = \frac{u_i
  +\frac{2}{u_i}}{2}$$ for all $i \ge 1.$
Then the function $F(n) = u_n$ in $n$ is not a rational function
over $\mathbb{Q}.$
\end{example}
\begin{proof} Since $u_{i+1} = \frac{u_i
  +\frac{2}{u_i}}{2} \ge \sqrt{u_i.\frac{2}{u_i}}=\sqrt{2} $ for all $i \ge
  1,$ it follows that
$$u_{i+1} = \frac{u_i
  +\frac{2}{u_i}}{2} \le \frac{u_i +\sqrt{2}}{2} \le u_i$$ for all $i \ge
  2.$ So $(u_n)_{n \ge 1}$ is convergent. Hence there  exists $\lim_{n \rightarrow
  \infty}u_n$.
  Set $$\lim_{n \rightarrow \infty}u_n
  =x.$$ Then $x = \frac{x
  +\frac{2}{x}}{2}$ since $u_{i+1} = \frac{u_i
  +\frac{2}{u_i}}{2}.$ Consequently,  $x = \sqrt{2}.$
   Since $\lim_{n \rightarrow \infty}u_n = \sqrt{2}$  is
an irrational number, the function $F(n) = u_n$ in $n$ is not a
rational function over $\mathbb{Q}$ by Theorem \ref {th2.1}(iii).

\end{proof}

%\noindent {\bf Example 3.3:} 
\begin{example}
$F(n) =
\sum_{j=1}^n\frac{(-1)^{j-1}}{j};$ $L(n) =
\sum_{j=1}^n\frac{(-1)^{j-1}}{2j-1};$ $S(n) =
\sum_{j=1}^n\frac{1}{j!}$ are  transcendental functions in $n$
over $\overline{\mathbb{Q}}.$ 
\end{example}

\begin{proof} Remember that  $\sum_{j=1}^\infty\frac{(-1)^{j-1}}{j}= \ln
2$ (see e.g. \cite{AK}).
 So $\lim_{n \rightarrow \infty} F(n) =
 \ln 2$.   Since $2 \in \mathbb{Q}\backslash\{0, 1\}$, $\ln 2$  is  transcendental  by the
Hermite-Lindemann theorem (see e.g. \cite[Appendix 1] {SL}). Hence
$F(n)$ is
 transcendental over $\overline{\mathbb{Q}}$ by Corollary
\ref{co2.4}.
  By the Leibnitz formula for $\pi,$ we have
$$\lim_{n \rightarrow \infty} L(n) =
\sum_{j=1}^\infty\frac{(-1)^{j-1}}{2j-1} = \frac{\pi}{4}$$ (see
e.g. \cite{BD, RG, AK, WD}).  Since $\pi$ is a transcendental
number by the Lindemann  theorem (see e.g. \cite{BA, SL, MK}),
$L(n)$ is a transcendental function over $\overline{\mathbb{Q}}$ by Corollary
\ref{co2.4}.
 Since $$1+\lim_{n \rightarrow \infty} S(n) = \sum_{j=0}^\infty\frac{1}{j!} =
e$$(see e.g. \cite{AK}) is a transcendental number by the Hermite
theorem (see e.g. \cite{BA,SL, MK}), $$S(n) =
\sum_{j=1}^n\frac{1}{j!}$$ is a transcendental function in $n$
over $\overline{\mathbb{Q}}$ by Corollary \ref{co2.4}. \end{proof}

%\noindent {\bf Example 3.4:} 
\begin{example} 
$U(n) = \sum_{j=1}^n\frac{1}{j^2}$
is a transcendental function in $n$ over $\overline{\mathbb{Q}}.$
\end{example}

\begin{proof}
Recall that the Basel problem asks for the precise summation of
the reciprocals of the squares of the natural numbers, i.e.,  the
precise sum of the infinite series:
$$\sum_{j=1}^\infty\frac{1}{j^2}.$$ This problem is a famous problem
in mathematical analysis with relevance to number theory, first
posed by Pietro Mengoli in 1644 and solved by Leonhard Euler in
1735  that $\sum_{j=1}^\infty\frac{1}{j^2} = \frac{\pi^2}{6}$ (see
e.g. \cite {EM, WA}). So $$\lim_{n \rightarrow \infty}U(n) =
\sum_{j=1}^\infty\frac{1}{j^2} = \frac{\pi^2}{6}.$$   Since $\pi$
is a transcendental number by the Lindemann theorem (see e.g.
\cite{BA, SL, MK}), $\frac{\pi^2}{6}$ is also a transcendental
number. Hence $U(n) = \sum_{j=1}^n\frac{1}{j^2}$ is a
transcendental function in $n$ over $\overline{\mathbb{Q}}$ by Corollary
\ref{co2.4}.
\end{proof}

\noindent{\bf Note 2:} Recall that $\psi(x)$ is the digamma function, that is, the logarithmic derivative of Euler's $\Gamma$-function. Then it is known that $\sum_{j=1}^n\frac{1}{j^2}= \frac{\pi^2}{6}-\psi'(n+1)$
 (see e.g. \cite{MN}).
 Hence $$\lim_{n \rightarrow \infty}[\frac{\pi^2}{6}-\psi'(n)] =
\sum_{j=1}^\infty\frac{1}{j^2} = \frac{\pi^2}{6}.$$  
So $\psi'(n)-\frac{\pi^2}{6}$ is a
transcendental function in $n$ over $\overline{\mathbb{Q}}$ 
by Corollary \ref{co2.2}. Now, since $\psi'(n+1)= \frac{\pi^2}{6}- \sum_{j=1}^n\frac{1}{j^2},$ it follows that
$\psi'(k)$ takes transcendental values for all $k$ positive
integer. 
Hence $\psi'(n)$ is a
transcendental function in $n$ over $\overline{\mathbb{Q}}.$


\begin{thebibliography}{99}{\small
\bibitem{BA} A. Baker, {\it  Transcendental Number Theory}, Cambridge
University Press 1975, ISBN 0-521-20461-5, Zbl 0297. 10013.


\bibitem{BD} J. Borwein, D. Bailey and R. Girgensohn,
{\it Experimentation in Mathematics-Computational Paths to
Discovery}, A K Peters 2003, ISBN 1-56881-136-5.

\bibitem{EM} W. Dunham, {\it Euler: The Master of Us All},
Mathematical Association of America  1999, ISBN 0-88385-328-0.

\bibitem{RG} R. C. Gupta, {\it On the remainder term in the Madhava-Leibniz's series},
Ganita Bharati 14(1992)(1-4), 68-71.

\bibitem{AK} A. W. Knapp, {\it Basic real analysis: along with a
companion volume Advanced Real Analysis}, ISBN0-8176-3250-6
(2005), QA300. K56.

\bibitem{SL} S. Lang, {\it Algebra,} Addison-Wesley Publishing Company (1965).

\bibitem{MK} K. Mahler, {\it Lectures on Transcendental Numbers},  Lecture
Notes in Mathematics 546, Springer-Verlag 1976, ISBN
3-540-07986-6, Zbl 0332. 10019.

\bibitem{PM} P. J. McCarthy, {\it Algebraic extensions of fields},
Dover Publications 1991, ISBN 0-486-66651-4.

\bibitem{MN} M. R. Murty and  N. Saradha, {\it Transcendental values of the digamma function}, J. Number theory, 125(2007), 298-318.   


\bibitem{JS1} J. Sondow,  {\it An antisymmetric formula for Euler's
constant}, Mathematics Magazine 71(1998), 219-220.

\bibitem{JS2} J. Sondow, {\it Criteria for irrationality of Euler's constant},
Proc. Amer. Math. Soc. 131(2003), 3335-3344.
\bibitem{WA} A. Weil, {\it Number Theory: An Approach Through History},
Springer-Verlag  1999, ISBN 0-8176-3141-0.

\bibitem{WD} D. Wells, {\it The Penguin Dictionary of Curious and Interesting
Numbers}, England: Penguin Books 1986.}
\end{thebibliography}
\end{document}